\documentclass[a4paper, 11pt]{article}
\usepackage[top=30truemm, bottom=30truemm, left=20truemm, right=20truemm]{geometry}
\usepackage{amsthm}
\usepackage{amsmath}
\usepackage{amssymb}
\usepackage{amsfonts}
\usepackage{mathtools}
\usepackage{bm}
\usepackage{color}
\usepackage[hidelinks]{hyperref}
\newtheoremstyle{mystyle}
	{\baselineskip}
	{\baselineskip}
	{\itshape}
	{}
	{\bfseries}
	{.}
	{1em}
	{}

\theoremstyle{mystyle}
	\newtheorem{thm}{Theorem}[section]
	\newtheorem{lem}[thm]{Lemma}
	
	\newtheorem{prop}[thm]{Proposition}

\newtheoremstyle{mystyle2}
	{\baselineskip}
	{\baselineskip}
	{\normalfont}
	{}
	{\bfseries}
	{.}
	{1em}
	{}

\theoremstyle{mystyle2}
	\newtheorem{rem}[thm]{Remark}
	
\newcommand{\C}{\mathbb{C}}
\newcommand{\R}{\mathbb{R}}
\newcommand{\Z}{\mathbb{Z}}
\newcommand{\FF}{\mathcal{F}}
\renewcommand{\SS}{\mathcal{S}}
\newcommand{\abs}[1]{\left\lvert #1 \right\rvert}
\newcommand{\norm}[1]{\left\lVert #1 \right\rVert}
\DeclareMathOperator{\re}{Re}
\DeclareMathOperator{\im}{Im}

\numberwithin{equation}{section}
\allowdisplaybreaks[4]
\begin{document}
\title{
	Remarks on the commutation relations between the Gauss--Weierstrass semigroup and monomial weights
}

\author{
	Yi C. Huang$^{1}$, Ryunosuke Kusaba$^{*2}$, and Tohru Ozawa$^{3}$ \\
\bigskip \\
	$^{1}$School of Mathematical Sciences, \\
	Nanjing Normal University, Nanjing 210023, China
\bigskip \\
	$^{2}$Department of Pure and Applied Physics, \\
	Graduate School of Advanced Science and Engineering, \\
	Waseda University, 3-4-1 Okubo, Shinjuku-ku, Tokyo 169-8555, Japan
\bigskip \\
	$^{3}$Department of Applied Physics, \\
	Waseda University, 3-4-1 Okubo, Shinjuku-ku, Tokyo 169-8555, Japan
}

\date{}

\footnotetext{E-mail: \texttt{Yi.Huang.Analysis@gmail.com} (Yi C. Huang), \texttt{ryu2411501@akane.waseda.jp} (Ryunosuke Kusaba), \\ \texttt{txozawa@waseda.jp} (Tohru Ozawa)}
\footnotetext{$^{*}$Corresponding author: Ryunosuke Kusaba}

\maketitle
\begin{abstract}
	We consider the commutation relations between the Gauss--Weierstrass semigroup on $\R^{n}$, generated by convolution with the complex Gaussian kernel, and monomial weights.
	We provide an explicit representation and a sharper estimate of the commutation relations with more concise proofs than those in previous works.
\end{abstract}




\newpage
\section{Introduction}

The Gauss--Weierstrass semigroup on $\R^{n}$ is defined by
\begin{align*}
	e^{\omega \Delta} \varphi \coloneqq G_{\omega} \ast \varphi
\end{align*}
for $\omega \in \C \setminus \left\{ 0 \right\}$ with $\re \omega \geq 0$ and $\varphi \in L^{q} \left( \R^{n} \right)$ with $q \in \left[ 1, + \infty \right]$, where $G_{\omega} \colon \R^{n} \to \C$ is the complex Gaussian kernel given by
\begin{align*}
	G_{\omega} \left( x \right) \coloneqq \left( 4 \pi \omega \right)^{- \frac{n}{2}} \exp \left( - \frac{\abs{x}^{2}}{4 \omega} \right), \qquad x \in \R^{n}
\end{align*}
and $\ast$ is the convolution in $\R^{n}$.
For $q \in \left[ 1, + \infty \right]$, let $L^{q} \left( \R^{n} \right) =L^{q} \left( \R^{n}; \C \right)$ denote the standard Lebesgue space equipped with the norm $\norm{\,\cdot\,}_{q} \coloneqq \norm{\,\cdot\,}_{L^{q} \left( \R^{n} \right)}$.
For each $m \in \Z_{>0}$, we define the weighted $L^{q}$-space by
\begin{align*}
	L^{q}_{m} \left( \R^{n} \right) \coloneqq \left\{ \varphi \in L^{q} \left( \R^{n} \right); {\,} \text{$x^{\alpha} \varphi \in L^{q} \left( \R^{n} \right)$ for any $\alpha \in \Z_{\geq 0}^{n}$ with $\abs{\alpha} \leq m$} \right\},
\end{align*}
where $x^{\alpha} \varphi$ is the function $x \mapsto x^{\alpha} \varphi \left( x \right)$ on $\R^{n}$.
We note that
\begin{align*}
	L^{q}_{m} \left( \R^{n} \right) = \left\{ \varphi \in L^{q} \left( \R^{n} \right); {\,} \abs{x}^{m} \varphi \in L^{q} \left( \R^{n} \right) \right\},
\end{align*}
where $\abs{x}^{m} \varphi$ means the function $x \mapsto \abs{x}^{m} \varphi \left( x \right)$ on $\R^{n}$.

In 1979, Weissler \cite{Weissler} obtained the $L^{p}$-$L^{q}$ hypercontractive estimates of the Gauss--Weierstrass semigroup with precise operator norms and Gaussian extremizers, subject to certain restrictions on $p$ and $q$ (see \cite[Theorem 3]{Weissler}).
For recent developments in harmonic analysis related to the so-called Weissler conjecture, we refer to \cite{Ivanisvili-Nazarov}.
The Gauss--Weierstrass semigroup also appears in the context of partial differential equations as semigroups associated with time evolution equations such as the nonlinear Schr\"{o}dinger equation ($\omega =it$ with $t \in \R$):
\begin{align}
	\label{Schrodinger}
	i \partial_{t} u+ \Delta u= \lambda \abs{u}^{p-1} u, \qquad \left( t, x \right) \in \R \times \R^{n},
\end{align}
where $\lambda \in \R$, and the complex Ginzburg--Landau equation ($\omega =t \nu$ with $t>0$ and $\re \nu >0$):
\begin{align}
	\label{CGL}
	\partial_{t} u- \nu \Delta u= \mu u+ \lambda \abs{u}^{p-1} u, \qquad \left( t, x \right) \in \left( 0, + \infty \right) \times \R^{n},
\end{align}
where $\mu \in \R$ and $\lambda \in \C$.
We remark that the complex Ginzburg--Landau equation includes the nonlinear heat equation ($\omega =t>0$) as a special case:
\begin{align}
	\label{Heat}
	\partial_{t} u- \Delta u= \lambda \abs{u}^{p-1} u, \qquad \left( t, x \right) \in \left( 0, + \infty \right) \times \R^{n},
\end{align}
where $\lambda \in \R$.
These equations have a broad background in both physics and mathematics.
We refer to \cite{Sulem-Sulem, Cazenave} for the nonlinear Schr\"{o}dinger equation, \cite{Aranson-Kramer, Duong-Nouaili-Zaag, Kusaba-Ozawa} for the complex Ginzburg--Landau equation, \cite{Giga-Giga-Saal, Quittner-Souplet} for the nonlinear heat equation, and the references therein.

In this paper, we focus on the commutation relations between the Gauss--Weierstrass semigroup and monomial weights:
\begin{align*}
	\left[ x^{\alpha}, e^{\omega \Delta} \right] \varphi \coloneqq x^{\alpha} e^{\omega \Delta} \varphi -e^{\omega \Delta} \left( x^{\alpha} \varphi \right)
\end{align*}
for $\alpha \in \Z_{\geq 0}^{n}$.
From the perspective of operator theory, the above commutation relation measures the difference between the semigroup $e^{\omega \Delta}$ and its perturbation $\left( 1/x^{\alpha} \right) e^{\omega \Delta} x^{\alpha}$, or more conveniently, its ``dominant'' inhomogeneous version $\left( 1+ \lvert x \rvert^{2} \right)^{- \abs{\alpha} /2} e^{\omega \Delta} \left( 1+ \lvert x \rvert^{2} \right)^{\abs{\alpha} /2}$.
This perturbation approach is widely used in establishing various smoothing estimates of the heat semigroup and has a huge impact on problems in harmonic analysis, such as the Calder\'{o}n--Zygmund theory for non-integral operators.
See \cite{Davies, Blunck-Kunstmann} for example.
From the viewpoint of partial differential equations, the commutation relations are useful for investigating properties of solutions.
If $\alpha =e_{j}$, where $\left( e_{j}; j=1, \ldots, n \right)$ is the standard basis of $\R^{n}$, then a straightforward computation yields
\begin{align*}
	\left[ x_{j}, e^{\omega \Delta} \right] \varphi =-2 \omega \partial_{j} e^{\omega \Delta} \varphi.
\end{align*}
From the above identity, we can expect that the commutation relations provide the connection between the decay of the initial data and the semigroup at the far field and the regularity of the semigroup.
In the case where $\omega =it$ with $t \in \R$, Jensen \cite{Jensen} employed the commutation relations to obtain precise information on the smoothing effect of the Schr\"{o}dinger evolution group in the weighted Sobolev spaces.
See also \cite{Ozawa1990, Ozawa1991} and the references therein.
As for the case where $\omega =t \nu$ with $t>0$ and $\re \nu >0$, in \cite{Kusaba-Ozawa}, the second and third authors of this paper derived an explicit representation of the commutation relations in terms of a linear combination of the semigroup and its derivatives.
Furthermore, they established a commutator estimate that refines the standard weighted estimate:
\begin{align*}
	\sum_{\abs{\alpha} =m} \norm{x^{\alpha} e^{t \nu \Delta} \varphi}_{q} \leq C \left( \norm{\abs{x}^{m} \varphi}_{q} +t^{\frac{m}{2}} \norm{\varphi}_{q} \right).
\end{align*}
As an application, they presented a new approach to deriving a weighted estimate of global solutions to \eqref{CGL} with $\mu =0$ in the super Fujita-critical case: $p>1+2/n$ (see Appendix \ref{app:CGL_weight}).
The weighted estimate of the global solutions plays an important role in investigating their large-time asymptotic behavior.
For details on this topic, we refer to \cite{Ishige-Kawakami2013, Ishige-Kawakami-Kobayashi, Ishige-Kawakami2024, Kusaba-Ozawa} for example.
In this paper, we concentrate on the case where $\re \omega >0$ and refine the results in \cite{Kusaba-Ozawa} on the representation of the commutation relations and their estimates.

\begin{thm} \label{th:commutator}
	Let $\omega \in \C$ satisfy $\re \omega >0$.
	Let $q \in \left[ 1, + \infty \right]$, $m \in \Z_{>0}$, and $\alpha \in \Z_{\geq 0}^{n}$ with $\abs{\alpha} =m$.
	Then, $x^{\alpha} e^{\omega \Delta} \varphi \in L^{q} \left( \R^{n} \right)$ holds for any $\varphi \in L_{m}^{q} \left( \R^{n} \right)$ with the identity
	\begin{align*}
		\left[ x^{\alpha}, e^{\omega \Delta} \right] \varphi =R_{\alpha} \left( \omega \right) \varphi
	\end{align*}
	in $L^{q} \left( \R^{n} \right)$, where
	\begin{align*}
		R_{\alpha} \left( \omega \right) \varphi \coloneqq \sum_{\substack{\beta + \gamma = \alpha \\ \beta \neq 0}} \sum_{2 \kappa \leq \beta} \frac{\alpha !}{\gamma ! \kappa ! \left( \beta -2 \kappa \right) !} \omega^{\abs{\kappa}} \left( -2 \omega \partial \right)^{\beta -2 \kappa} e^{\omega \Delta} \left( x^{\gamma} \varphi \right).
	\end{align*}
\end{thm}

In \cite{Kusaba-Ozawa}, it was shown that $R_{\alpha} \left( \omega \right) \varphi$ is given by
\begin{align*}
	R_{\alpha} \left( \omega \right) \varphi = \sum_{\substack{\beta + \gamma = \alpha \\ \beta \neq 0}} \frac{\alpha !}{\beta ! \gamma !} \left( -2 \omega \partial \right)^{\beta} e^{\omega \Delta} \left( x^{\gamma} \varphi \right) + \sum_{\substack{\beta + \gamma \leq \alpha, {\ } \abs{\beta + \gamma} \leq \abs{\alpha} -2 \\ \abs{\beta} +1 \leq \ell \leq \frac{\abs{\alpha} + \abs{\beta} - \abs{\gamma}}{2}}} C_{\ell \beta \gamma}^{\alpha} \omega^{\ell} \partial^{\beta} e^{\omega \Delta} \left( x^{\gamma} \varphi \right)
\end{align*}
with some implicit constants $C_{\ell \beta \gamma}^{\alpha} \in \R$ independent of $\omega$.
Theorem \ref{th:commutator} implies that the second sum in the above representation can be unified into the first sum with explicit constants.

We prove Theorem \ref{th:commutator} by using the binomial expansion.
In addition, we provide an alternative proof for the case where $q \neq + \infty$, employing the Fourier transform and a density argument.
In both proofs, the multi-variable Hermite polynomials with the complex parameter $\omega$ play a crucial role.
We remark that the Hermite polynomials are also useful for studying the smoothing effect of the Schr\"{o}dinger evolution group (cf. \cite{Ozawa1990}).

\begin{thm} \label{th:commutator_esti}
	Let $\omega \in \C$ satisfy $\re \omega >0$.
	Let $1 \leq q \leq p \leq + \infty$ and let $m \in \Z_{>0}$.
	Then, the estimate
	\begin{align*}
		\sum_{\abs{\alpha} =m} \norm{\left[ x^{\alpha}, e^{\omega \Delta} \right] \varphi}_{p} &\leq A_{m, r} \left( \theta \right) \abs{\omega}^{- \frac{n}{2} \left( \frac{1}{q} - \frac{1}{p} \right)} \left( \abs{\omega}^{\frac{1}{2}} \norm{\lvert x \rvert^{m-1} {\,} \varphi}_{q} + \abs{\omega}^{\frac{m}{2}} \norm{\varphi}_{q} \right)
	\end{align*}
	holds for any $\varphi \in L_{m}^{q} \left( \R^{n} \right)$, where
	\begin{align*}
		A_{m, r} \left( \theta \right) \coloneqq \frac{\left( n+m-1 \right) !}{\left( n-1 \right) ! m!} \left( 4 \pi \right)^{- \frac{n}{2} \left( 1- \frac{1}{r} \right)} \left( \frac{2}{r \cos \theta} \right)^{\frac{n}{2r}} \left\{ \left[ \left( \frac{4m}{e \cos \theta} \right)^{\frac{1}{2}} +1 \right]^{m} -1 \right\},
	\end{align*}
	$\theta \in \left( - \pi /2, \pi /2 \right)$ with $e^{i \theta} = \omega / \abs{\omega}$, and $r \in \left[ 1, + \infty \right]$ with $1/p+1=1/r+1/q$.
\end{thm}

In \cite{Kusaba-Ozawa}, it was proved that the estimate
\begin{align*}
	\sum_{\abs{\alpha} =m} \norm{\left[ x^{\alpha}, e^{\omega \Delta} \right] \varphi}_{1} &\leq C \left( \abs{\omega}^{\frac{1}{2}} \norm{\lvert x \rvert^{m-1} {\,} \varphi}_{1} + \left( \abs{\omega}^{\frac{1}{2}} + \abs{\omega}^{\frac{m}{2}} \right) \norm{\varphi}_{1} \right)
\end{align*}
holds for any $\varphi \in L_{m}^{1} \left( \R^{n} \right)$, where $C>0$ is a constant independent of $\abs{\omega}$.
Theorem \ref{th:commutator_esti} enables us to eliminate the term $\abs{\omega}^{\frac{1}{2}} \norm{\varphi}_{1}$ in the above estimate, which is reasonable from the viewpoint of scaling.
Furthermore, we explicitly determine how the operator norm bound $C$ depends on the parameter $\omega$.
This parameter dependence is important in the scenarios of considering the vanishing limits as $\re \omega \to 0$ or $\im \omega \to 0$.
We emphasize that $A_{m, r} \left( \theta \right)$ is independent of $\sin \theta$ and in particular, that $A_{m, r} \left( \theta \right) \to + \infty$ as $\abs{\theta} \to \pi /2$.
This fact naturally leads us to employ the weighted Sobolev spaces in the case where $\re \omega =0$, namely, the case where $\abs{\theta} = \pi /2$ as in \cite{Jensen, Ozawa1990, Ozawa1991}.

\begin{rem}
	An estimate similar to that in Theorem \ref{th:commutator_esti} is true if the monomial weights $x^{\alpha}$ with $\abs{\alpha} =m$ are replaced by the positive and monotonic weight $\abs{x}^{m}$.
	In fact, under the same assumption as in Theorem \ref{th:commutator_esti}, the estimate
	\begin{align*}
		\norm{\left[ \abs{x}^{m}, e^{\omega \Delta} \right] \varphi}_{p} &\leq \widetilde{A}_{m, r} \left( \theta \right) \abs{\omega}^{- \frac{n}{2} \left( \frac{1}{q} - \frac{1}{p} \right)} \left( \abs{\omega}^{\frac{1}{2}} \norm{\lvert x \rvert^{m-1} {\,} \varphi}_{q} + \abs{\omega}^{\frac{m}{2}} \norm{\varphi}_{q} \right)
	\end{align*}
	holds for any $\varphi \in L_{m}^{q} \left( \R^{n} \right)$, where
	\begin{align*}
		\widetilde{A}_{m, r} \left( \theta \right) \coloneqq \left( 4 \pi \right)^{- \frac{n}{2} \left( 1- \frac{1}{r} \right)} \left( \frac{2}{r \cos \theta} \right)^{\frac{n}{2r}} \left\{ \left[ \left( \frac{4m}{e \cos \theta} \right)^{\frac{1}{2}} +1 \right]^{m} -1 \right\},
	\end{align*}
	$\theta \in \left( - \pi /2, \pi /2 \right)$ with $e^{i \theta} = \omega / \abs{\omega}$, and $r \in \left[ 1, + \infty \right]$ with $1/p+1=1/r+1/q$.
	However, it seems impossible to represent $\left[ \abs{x}^{m}, e^{\omega \Delta} \right] \varphi$ as a linear combination of the Gauss--Weierstrass semigroup and its derivatives.
\end{rem}

The rest of this paper is organized as follows.
In the next section, we recall the fundamental properties of the multi-variable Hermite polynomials with the complex parameter $\omega$.
In Sections \ref{sec:commutator} and \ref{sec:commutator_esti}, we give the proofs of Theorems \ref{th:commutator} and \ref{th:commutator_esti}, respectively.

\section{Hermite polynomials} \label{sec:Hermite}

We first recall the multi-index notation.
Let $\Z_{>0}$ be the set of all positive integers and let $\Z_{\geq 0} \coloneqq \Z_{>0} \cup \left\{ 0 \right\}$.
For $\alpha = \left( \alpha_{1}, \ldots, \alpha_{n} \right) \in \Z_{\geq 0}^{n}$, $\beta = \left( \beta_{1}, \ldots, \beta_{n} \right) \in \Z_{\geq 0}^{n}$, and $\lambda \in \Z_{\geq 0}$, we define $\alpha \pm \beta$ and $\lambda \alpha$ as follows:
\begin{align*}
	\alpha \pm \beta &\coloneqq \left( \alpha_{1} \pm \beta_{1}, \ldots, \alpha_{n} \pm \beta_{n} \right), \\
	\lambda \alpha &\coloneqq \left( \lambda \alpha_{1}, \ldots, \lambda \alpha_{n} \right).
\end{align*}
In addition, $\beta \leq \alpha$ means that $\beta_{j} \leq \alpha_{j}$ holds for any $j \in \left\{ 1, \ldots, n \right\}$.
We remark that $\alpha + \beta$, $\lambda \alpha \in \Z_{\geq 0}^{n}$ for all $\alpha, \beta \in \Z_{\geq 0}^{n}$ and $\lambda \in \Z_{\geq 0}$, while $\alpha - \beta \in \Z_{\geq 0}^{n}$ if and only if $\beta \leq \alpha$.
For $\alpha = \left( \alpha_{1}, \ldots, \alpha_{n} \right) \in \Z_{\geq 0}^{n}$ and $x= \left( x_{1}, \ldots, x_{n} \right) \in \R^{n}$, we set
\begin{align*}
	\abs{\alpha} \coloneqq \sum_{j=1}^{n} \alpha_{j}, \qquad \alpha ! \coloneqq \prod_{j=1}^{n} \alpha_{j} !, \qquad x^{\alpha} \coloneqq \prod_{j=1}^{n} x_{j}^{\alpha_{j}}, \qquad \partial^{\alpha} = \partial_{x}^{\alpha} \coloneqq \prod_{j=1}^{n} \partial_{j}^{\alpha_{j}}, \qquad \partial_{j} \coloneqq \frac{\partial}{\partial x_{j}}.
\end{align*}
For $j \in \left\{ 1, \ldots, n \right\}$, let $e_{j} \in \Z_{\geq 0}^{n}$ denote a multi-index with $\abs{e_{j}} =1$ whose components are $0$ except for the $j$th coordinate.

Let $\omega \in \C \setminus \left\{ 0 \right\}$ satisfy $\re \omega \geq 0$.
To rewrite the derivatives of the complex Gaussian kernel $G_{\omega}$ explicitly, we next introduce the multi-variable Hermite polynomials with the complex parameter $\omega$.
As for the basic properties of the Hermite polynomials, we refer to \cite[Section 1.4]{Thangavelu} and the references therein.
For each $\alpha \in \Z_{\geq 0}^{n}$, we define the Hermite polynomial of order $\alpha$ by
\begin{align*}
	\bm{H}_{\omega, \alpha} \left( x \right) \coloneqq \left( -1 \right)^{\abs{\alpha}} \exp \left( \frac{\abs{x}^{2}}{\omega} \right) \partial^{\alpha} \exp \left( - \frac{\abs{x}^{2}}{\omega} \right), \qquad x \in \R^{n}.
\end{align*}
By a simple computation, we obtain
\begin{align*}
	\bm{H}_{\omega, \alpha} \left( x \right) = \sum_{2 \beta \leq \alpha} \frac{\left( -1 \right)^{\abs{\beta}} \alpha !}{\beta ! \left( \alpha -2 \beta \right) !} \omega^{- \abs{\alpha - \beta}} \left( 2x \right)^{\alpha -2 \beta}.
\end{align*}
In addition, we see that
\begin{align*}
	\left( \partial^{\alpha} G_{\omega} \right) \left( x \right)
	= \left( -2 \right)^{- \abs{\alpha}} \bm{H}_{\omega, \alpha} \left( \frac{x}{2} \right) G_{\omega} \left( x \right).
\end{align*}
Hence, by setting
\begin{align*}
	\bm{h}_{\omega, \alpha} \left( x \right) \coloneqq \bm{H}_{\omega, \alpha} \left( \frac{x}{2} \right) = \sum_{2 \beta \leq \alpha} \frac{\left( -1 \right)^{\abs{\beta}} \alpha !}{\beta ! \left( \alpha -2 \beta \right) !} \omega^{- \abs{\alpha - \beta}} x^{\alpha -2 \beta}, \qquad x \in \R^{n},
\end{align*}
we can rewrite $\partial^{\alpha} G_{\omega}$ as
\begin{align}
	\label{eq:Gauss_derivative}
	\left( \partial^{\alpha} G_{\omega} \right) \left( x \right) = \left( -2 \right)^{- \abs{\alpha}} \bm{h}_{\omega, \alpha} \left( x \right) G_{\omega} \left( x \right).
\end{align}

The following lemma plays an important role in the proof of Theorem \ref{th:commutator}.

\begin{lem} \label{lem:Hermite_monomial}
	Let $\alpha \in \Z_{\geq 0}^{n}$.
	Then, the identity
	\begin{align*}
		x^{\alpha} = \sum_{2 \beta \leq \alpha} \frac{\alpha !}{\beta ! \left( \alpha -2 \beta \right) !} \omega^{\abs{\alpha - \beta}} \bm{h}_{\omega, \alpha -2 \beta} \left( x \right)
	\end{align*}
	holds for any $x \in \R^{n}$.
\end{lem}

The above lemma provides the inverse representation of the Hermite polynomials, namely, the expansion of monomials in terms of the Hermite polynomials.
In the case where $\omega =1$, this representation is relatively well known.
However, even in this case, we could not find a reference that gives a rigorous proof.
For the convenience of the reader, we present the proof of Lemma \ref{lem:Hermite_monomial} in Appendix \ref{app:lem_Hermite}.

\section{Proof of Theorem \ref{th:commutator}}
\label{sec:commutator}

We first give the proof of Theorem \ref{th:commutator} with the aid of the binomial expansion.

\begin{proof}[Proof of Theorem \ref{th:commutator}]
	Let $\varphi \in L^{q}_{m} \left( \R^{n} \right)$.
	We note that $x^{\beta} G_{\omega} \in L^{1} \left( \R^{n} \right)$ for any $\beta \in \Z_{\geq 0}^{n}$ since $\re \omega >0$.
	For $\beta, \gamma \in \Z_{\geq 0}^{n}$ with $\abs{\gamma} \leq m$, it follows from Young's inequality that
	\begin{align*}
		\lVert ( x^{\beta} G_{\omega} ) \ast \left( x^{\gamma} \varphi \right) \rVert_{q} &\leq \lVert x^{\beta} G_{\omega} \rVert_{1} {\,} \lVert x^{\gamma} \varphi \rVert_{q} <+ \infty,
	\end{align*}
	which implies $( x^{\beta} G_{\omega} ) \ast \left( x^{\gamma} \varphi \right) \in L^{q} \left( \R^{n} \right)$.
	By the binomial expansion, we have
	\begin{align*}
		\sum_{\beta + \gamma = \alpha} \frac{\alpha !}{\beta ! \gamma !} \left( ( x^{\beta} G_{\omega} ) \ast \left( x^{\gamma} \varphi \right) \right) \left( x \right) &= \sum_{\beta + \gamma = \alpha} \frac{\alpha !}{\beta ! \gamma !} \int_{\R^{n}} \left( x-y \right)^{\beta} G_{\omega} \left( x-y \right) y^{\gamma} \varphi \left( y \right) dy \\
		&= \int_{\R^{n}} \left( \left( x-y \right) +y \right)^{\alpha} G_{\omega} \left( x-y \right) \varphi \left( y \right) dy \\
		&=x^{\alpha} \int_{\R^{n}} G_{\omega} \left( x-y \right) \varphi \left( y \right) dy \\
		&=x^{\alpha} \left( e^{\omega \Delta} \varphi \right) \left( x \right).
	\end{align*}
	Therefore, we deduce that $x^{\alpha} e^{\omega \Delta} \varphi \in L^{q} \left( \R^{n} \right)$ and
	\begin{align*}
		\left[ x^{\alpha}, e^{\omega \Delta} \right] \varphi = \sum_{\substack{\beta + \gamma = \alpha \\ \beta \neq 0}} \frac{\alpha !}{\beta ! \gamma !} {\,} ( x^{\beta} G_{\omega} ) \ast \left( x^{\gamma} \varphi \right).
	\end{align*}
	Moreover, by virtue of Lemma \ref{lem:Hermite_monomial} and \eqref{eq:Gauss_derivative}, we obtain
	\begin{align*}
		&\sum_{\substack{\beta + \gamma = \alpha \\ \beta \neq 0}} \frac{\alpha !}{\beta ! \gamma !} \left( ( x^{\beta} G_{\omega} ) \ast \left( x^{\gamma} \varphi \right) \right) \left( x \right) \\
		&\hspace{1cm} = \sum_{\substack{\beta + \gamma = \alpha \\ \beta \neq 0}} \frac{\alpha !}{\beta ! \gamma !} \int_{\R^{n}} \left( x-y \right)^{\beta} G_{\omega} \left( x-y \right) y^{\gamma} \varphi \left( y \right) dy \\
		&\hspace{1cm} = \sum_{\substack{\beta + \gamma = \alpha \\ \beta \neq 0}} \sum_{2 \kappa \leq \beta} \frac{\alpha !}{\beta ! \gamma !} \frac{\beta !}{\kappa ! \left( \beta -2 \kappa \right) !} \omega^{\abs{\beta - \kappa}} \int_{\R^{n}} \bm{h}_{\omega, \beta -2 \kappa} \left( x-y \right) G_{\omega} \left( x-y \right) y^{\gamma} \varphi \left( y \right) dy \\
		&\hspace{1cm} = \sum_{\substack{\beta + \gamma = \alpha \\ \beta \neq 0}} \sum_{2 \kappa \leq \beta} \frac{\alpha !}{\gamma ! \kappa ! \left( \beta -2 \kappa \right) !} \left( -2 \right)^{\abs{\beta -2 \kappa}} \omega^{\abs{\beta - \kappa}} \int_{\R^{n}} \left( \partial^{\beta -2 \kappa} G_{\omega} \right) \left( x-y \right) y^{\gamma} \varphi \left( y \right) dy \\
		&\hspace{1cm} = \sum_{\substack{\beta + \gamma = \alpha \\ \beta \neq 0}} \sum_{2 \kappa \leq \beta} \frac{\alpha !}{\gamma ! \kappa ! \left( \beta -2 \kappa \right) !} \omega^{\abs{\kappa}} \left( \left( -2 \omega \partial \right)^{\beta -2 \kappa} e^{\omega \Delta} \left( x^{\gamma} \varphi \right) \right) \left( x \right).
	\end{align*}
	This completes the proof of Theorem \ref{th:commutator}.
\end{proof}

We next prove Theorem \ref{th:commutator} with $q \neq + \infty$ by using the Fourier transform.
Here and hereafter, let $\SS \left( \R^{n} \right)$ denote the Schwartz space, namely,
\begin{align*}
	\SS \left( \R^{n} \right) \coloneqq \left\{ \varphi \in C^{\infty} \left( \R^{n} \right); {\,} \text{$\sup_{x \in \R^{n}} {\,} \lvert x^{\alpha} \partial^{\beta} \varphi \left( x \right) \rvert <+ \infty$ for all $\alpha, \beta \in \Z_{\geq 0}^{n}$} \right\}.
\end{align*}
We define the Fourier transform $\FF$ on $\SS \left( \R^{n} \right)$ by
\begin{align*}
	\left( \FF \varphi \right) \left( \xi \right) \coloneqq \left( 2 \pi \right)^{- \frac{n}{2}} \int_{\R^{n}} e^{-ix \cdot \xi} \varphi \left( x \right) dx, \qquad \xi \in \R^{n}
\end{align*}
for $\varphi \in \SS \left( \R^{n} \right)$, whose inverse is given by
\begin{align*}
	\left( \FF^{-1} \varphi \right) \left( x \right) = \left( 2 \pi \right)^{- \frac{n}{2}} \int_{\R^{n}} e^{ix \cdot \xi} \varphi \left( \xi \right) d \xi, \qquad x \in \R^{n}.
\end{align*}
Then, we see that the identities
\begin{align*}
	\left( i \partial \right)^{\alpha} \FF \varphi = \FF \left( x^{\alpha} \varphi \right), \qquad \xi^{\alpha} \FF \varphi = \FF \left( \left( -i \partial \right)^{\alpha} \varphi \right)
\end{align*}
hold for any $\varphi \in \SS \left( \R^{n} \right)$ and $\alpha \in \Z_{\geq 0}^{n}$.

\begin{proof}[Proof of Theorem \ref{th:commutator} with $q \neq + \infty$]
	Since $\SS \left( \R^{n} \right)$ is dense in $L^{q}_{m} \left( \R^{n} \right)$ if $q \neq + \infty$ and $e^{\omega \Delta}$ is continuous on $L^{q}_{m} \left( \R^{n} \right)$, it suffices to consider the case where $\varphi \in \SS \left( \R^{n} \right)$.
	We start with the fact that $e^{\omega \Delta} \varphi$ can be represented as
	\begin{align*}
		e^{\omega \Delta} \varphi = \FF^{-1} \left( \exp \left( - \omega \abs{\xi}^{2} \right) \FF \varphi \right).
	\end{align*}
	We remark that $e^{\omega \Delta} \varphi \in \SS \left( \R^{n} \right)$.
	By using the above representation, we have
	\begin{align*}
		\FF \left( x^{\alpha} e^{\omega \Delta} \varphi \right)
		&= \left( i \partial \right)^{\alpha} \FF \left( e^{\omega \Delta} \varphi \right) \\
		&= \left( i \partial \right)^{\alpha} \left( \exp \left( - \omega \abs{\xi}^{2} \right) \FF \varphi \right) \\
		&= \sum_{\beta + \gamma = \alpha} \frac{\alpha !}{\beta ! \gamma !} \left( i \partial \right)^{\beta} \exp \left( - \omega \abs{\xi}^{2} \right) \left( i \partial \right)^{\gamma} \FF \varphi.
	\end{align*}
	It follows from the definition of the Hermite polynomials that
	\begin{align*}
		\left( i \partial \right)^{\beta} \exp \left( - \omega \abs{\xi}^{2} \right)
		&= \left( -i \right)^{\abs{\beta}} \left[ \left( -1 \right)^{\abs{\beta}} \exp \left( \omega \abs{\xi}^{2} \right) \partial^{\beta} \exp \left( - \omega \abs{\xi}^{2} \right) \right] \exp \left( - \omega \abs{\xi}^{2} \right) \\
		&= \left( -i \right)^{\abs{\beta}} \bm{H}_{1/ \omega, \beta} \left( \xi \right) \exp \left( - \omega \abs{\xi}^{2} \right).
	\end{align*}
	Combining these identities yields
	\begin{align*}
		\FF \left( x^{\alpha} e^{\omega \Delta} \varphi \right)
		&= \sum_{\beta + \gamma = \alpha} \frac{\alpha !}{\beta ! \gamma !} \left( -i \right)^{\abs{\beta}} \bm{H}_{1/ \omega, \beta} \left( \xi \right) \exp \left( - \omega \abs{\xi}^{2} \right) \left( i \partial \right)^{\gamma} \FF \varphi \\
		&= \sum_{\beta + \gamma = \alpha} \frac{\alpha !}{\beta ! \gamma !} \left( -i \right)^{\abs{\beta}} \bm{H}_{1/ \omega, \beta} \left( \xi \right) \exp \left( - \omega \abs{\xi}^{2} \right) \FF \left( x^{\gamma} \varphi \right) \\
		&= \sum_{\beta + \gamma = \alpha} \frac{\alpha !}{\beta ! \gamma !} \left( -i \right)^{\abs{\beta}} \bm{H}_{1/ \omega, \beta} \left( \xi \right) \FF \left( e^{\omega \Delta} \left( x^{\gamma} \varphi \right) \right) \\
		&= \FF \left( \sum_{\beta + \gamma = \alpha} \frac{\alpha !}{\beta ! \gamma !} \left( -i \right)^{\abs{\beta}} \bm{H}_{1/ \omega, \beta} \left( -i \partial \right) e^{\omega \Delta} \left( x^{\gamma} \varphi \right) \right),
	\end{align*}
	which implies
	\begin{align*}
		x^{\alpha} e^{\omega \Delta} \varphi = \sum_{\beta + \gamma = \alpha} \frac{\alpha !}{\beta ! \gamma !} \left( -i \right)^{\abs{\beta}} \bm{H}_{1/ \omega, \beta} \left( -i \partial \right) e^{\omega \Delta} \left( x^{\gamma} \varphi \right).
	\end{align*}
	We note that
	\begin{align*}
		\left( -i \right)^{\abs{\beta}} \bm{H}_{1/ \omega, \beta} \left( -i \partial \right) &= \left( -i \right)^{\abs{\beta}} \sum_{2 \kappa \leq \beta} \frac{\left( -1 \right)^{\abs{\kappa}} \beta !}{\kappa ! \left( \beta -2 \kappa \right) !} \omega^{\abs{\beta - \kappa}} \left( -2i \partial \right)^{\beta -2 \kappa} \\
		&= \sum_{2 \kappa \leq \beta} \frac{\beta !}{\kappa ! \left( \beta -2 \kappa \right) !} \omega^{\abs{\kappa}} \left( -2 \omega \partial \right)^{\beta -2 \kappa}.
	\end{align*}
	Taking into account the fact that $\bm{H}_{1/ \omega, 0} \equiv 1$, we conclude that
	\begin{align*}
		x^{\alpha} e^{\omega \Delta} \varphi =e^{\omega \Delta} \left( x^{\alpha} \varphi \right) + \sum_{\substack{\beta + \gamma = \alpha \\ \beta \neq 0}} \sum_{2 \kappa \leq \beta} \frac{\alpha !}{\gamma ! \kappa ! \left( \beta -2 \kappa \right) !} \omega^{\abs{\kappa}} \left( -2 \omega \partial \right)^{\beta -2 \kappa} e^{\omega \Delta} \left( x^{\gamma} \varphi \right).
	\end{align*}
	This completes the proof of Theorem \ref{th:commutator}.
\end{proof}

\section{Proof of Theorem \ref{th:commutator_esti}}
\label{sec:commutator_esti}

Let $\varphi \in L_{m}^{q} \left( \R^{n} \right)$.
We note that the complex Gaussian kernel $G_{\omega}$ has the self-similarity described as
\begin{align*}
	G_{\omega} \left( x \right) = \abs{\omega}^{- \frac{n}{2}} G_{e^{i \theta}} \left( \abs{\omega}^{- \frac{1}{2}} x \right), \qquad x \in \R^{n},
\end{align*}
which yields
\begin{align*}
	\lVert x^{\beta} G_{\omega} \rVert_{r} = \abs{\omega}^{- \frac{n}{2} \left( 1- \frac{1}{r} \right) + \frac{\abs{\beta}}{2}} {\,} \lVert x^{\beta} G_{e^{i \theta}} \rVert_{r}
\end{align*}
for all $\beta \in \Z_{\geq 0}^{n}$.
We use the following representation of $R_{\alpha} \left( \omega \right) \varphi$, which is given in the proof of Theorem \ref{th:commutator} with the binomial expansion:
\begin{align}
	\label{eq:R_conv}
	R_{\alpha} \left( \omega \right) \varphi = \sum_{\substack{\beta + \gamma = \alpha \\ \beta \neq 0}} \frac{\alpha !}{\beta ! \gamma !} {\,} ( x^{\beta} G_{\omega} ) \ast \left( x^{\gamma} \varphi \right).
\end{align}
By Young's inequality, we have
\begin{align*}
	\sum_{\abs{\alpha} =m} \norm{R_{\alpha} \left( \omega \right) \varphi}_{p}
	&\leq \sum_{\abs{\alpha} =m} \sum_{\substack{\beta + \gamma = \alpha \\ \beta \neq 0}} \frac{\alpha !}{\beta ! \gamma !} {\,} \lVert ( x^{\beta} G_{\omega} ) \ast \left( x^{\gamma} \varphi \right) \rVert_{p} \\
	&\leq \sum_{\abs{\alpha} =m} \sum_{\substack{\beta + \gamma = \alpha \\ \beta \neq 0}} \frac{\alpha !}{\beta ! \gamma !} {\,} \lVert x^{\beta} G_{\omega} \rVert_{r} {\,} \lVert x^{\gamma} \varphi \rVert_{q} \\
	&= \sum_{\abs{\alpha} =m} \sum_{\substack{\beta + \gamma = \alpha \\ \beta \neq 0}} \frac{\alpha !}{\beta ! \gamma !} \abs{\omega}^{- \frac{n}{2} \left( 1- \frac{1}{r} \right) + \frac{\abs{\beta}}{2}} {\,} \lVert x^{\beta} G_{e^{i \theta}} \rVert_{r} {\,} \lVert x^{\gamma} \varphi \rVert_{q} \\
	&\leq \abs{\omega}^{- \frac{n}{2} \left( \frac{1}{q} - \frac{1}{p} \right)} \sum_{\abs{\alpha} =m} \sum_{\substack{\beta + \gamma = \alpha \\ \beta \neq 0}} \frac{\alpha !}{\beta ! \gamma !} \abs{\omega}^{\frac{\abs{\beta}}{2}} {\,} \lVert \abs{x}^{\abs{\beta}} G_{e^{i \theta}} \rVert_{r} {\,} \lVert \abs{x}^{\abs{\gamma}} \varphi \rVert_{q}.
\end{align*}
In the case where $m=1$, since $\alpha, \beta, \gamma \in \Z_{\geq 0}^{n}$ satisfy $\abs{\alpha} =1$, $\beta \neq 0$, and $\beta + \gamma = \alpha$ if and only if $\left( \beta, \gamma \right) = \left( \alpha, 0 \right)$, we obtain
\begin{align*}
	\sum_{\abs{\alpha} =1} \norm{R_{\alpha} \left( \omega \right) \varphi}_{p} &\leq \abs{\omega}^{- \frac{n}{2} \left( \frac{1}{q} - \frac{1}{p} \right) + \frac{1}{2}} \norm{\varphi}_{q} \sum_{\abs{\alpha} =1} \norm{\abs{x} G_{e^{i \theta}}}_{r}.
\end{align*}
If $m \geq 2$, then for any $\alpha, \beta, \gamma \in \Z_{\geq 0}$ satisfying $\abs{\alpha} =m$, $\abs{\beta} \geq 2$, and $\beta + \gamma = \alpha$, H\"{o}lder's inequality implies
\begin{align*}
	\abs{\omega}^{\frac{\abs{\beta}}{2}} {\,} \lVert \abs{x}^{\abs{\gamma}} \varphi \rVert_{q}
	&\leq \abs{\omega}^{\frac{\abs{\beta}}{2}} \norm{\lvert x \rvert^{m-1} {\,} \varphi}_{q}^{\frac{\abs{\gamma}}{m-1}} \norm{\varphi}_{q}^{\frac{\abs{\beta} -1}{m-1}} \\
	&= \abs{\omega}^{\frac{1}{2}} \norm{\lvert x \rvert^{m-1} {\,} \varphi}_{q}^{\frac{\abs{\gamma}}{m-1}} \left( \abs{\omega}^{\frac{m-1}{2}} \norm{\varphi}_{q} \right)^{\frac{\abs{\beta} -1}{m-1}} \\
	&\leq \abs{\omega}^{\frac{1}{2}} \left( \norm{\lvert x \rvert^{m-1} {\,} \varphi}_{q} + \abs{\omega}^{\frac{m-1}{2}} \norm{\varphi}_{q} \right) \\
	&= \abs{\omega}^{\frac{1}{2}} \norm{\lvert x \rvert^{m-1} {\,} \varphi}_{q} + \abs{\omega}^{\frac{m}{2}} \norm{\varphi}_{q},
\end{align*}
whence follows
\begin{align*}
	\sum_{\abs{\alpha} =m} \norm{R_{\alpha} \left( \omega \right) \varphi}_{p}
	&\leq \abs{\omega}^{- \frac{n}{2} \left( \frac{1}{q} - \frac{1}{p} \right) + \frac{1}{2}} \norm{\lvert x \rvert^{m-1} {\,} \varphi}_{q} \sum_{\abs{\alpha} =m} \sum_{\substack{\beta + \gamma = \alpha \\ \abs{\beta} =1}} \frac{\alpha !}{\beta ! \gamma !} \norm{\abs{x} G_{e^{i \theta}}}_{r} \\
	&\hspace{1cm} + \abs{\omega}^{- \frac{n}{2} \left( \frac{1}{q} - \frac{1}{p} \right)} \left( \abs{\omega}^{\frac{1}{2}} \norm{\lvert x \rvert^{m-1} {\,} \varphi}_{q} + \abs{\omega}^{\frac{m}{2}} \norm{\varphi}_{q} \right) \sum_{\abs{\alpha} =m} \sum_{\substack{\beta + \gamma = \alpha \\ \abs{\beta} \geq 2}} \frac{\alpha !}{\beta ! \gamma !} {\,} \lVert \abs{x}^{\abs{\beta}} G_{e^{i \theta}} \rVert_{r} \\
	&\leq \abs{\omega}^{- \frac{n}{2} \left( \frac{1}{q} - \frac{1}{p} \right)} \left( \abs{\omega}^{\frac{1}{2}} \norm{\lvert x \rvert^{m-1} {\,} \varphi}_{q} + \abs{\omega}^{\frac{m}{2}} \norm{\varphi}_{q} \right) \sum_{\abs{\alpha} =m} \sum_{\substack{\beta + \gamma = \alpha \\ \beta \neq 0}} \frac{\alpha !}{\beta ! \gamma !} {\,} \lVert \abs{x}^{\abs{\beta}} G_{e^{i \theta}} \rVert_{r}.
\end{align*}
Finally, by a simple computation, we get
\begin{align*}
	&\sum_{\abs{\alpha} =m} \sum_{\substack{\beta + \gamma = \alpha \\ \beta \neq 0}} \frac{\alpha !}{\beta ! \gamma !} {\,} \lVert \abs{x}^{\abs{\beta}} G_{e^{i \theta}} \rVert_{r} \\
	&\hspace{1cm} \leq 2^{\frac{n}{2}} \norm{G_{2e^{i \theta}}}_{r} \sum_{\abs{\alpha} =m} \sum_{\substack{\beta + \gamma = \alpha \\ \beta \neq 0}} \frac{\alpha !}{\beta ! \gamma !} \sup_{x \in \R^{n}} \left[ \abs{x}^{\abs{\beta}} \exp \left( - \frac{\cos \theta}{8} \abs{x}^{2} \right) \right] \\
	&\hspace{1cm} =2^{\frac{n}{2}} \left( 8 \pi \right)^{- \frac{n}{2} \left( 1- \frac{1}{r} \right)} \left( \frac{1}{r \cos \theta} \right)^{\frac{n}{2r}} \sum_{\abs{\alpha} =m} \sum_{\substack{\beta + \gamma = \alpha \\ \beta \neq 0}} \frac{\alpha !}{\beta ! \gamma !} \left( \frac{4 \abs{\beta}}{e \cos \theta} \right)^{\frac{\abs{\beta}}{2}} \\
	&\hspace{1cm} \leq \frac{\left( n+m-1 \right) !}{\left( n-1 \right) ! m!} \left( 4 \pi \right)^{- \frac{n}{2} \left( 1- \frac{1}{r} \right)} \left( \frac{2}{r \cos \theta} \right)^{\frac{n}{2r}} \left\{ \left[ \left( \frac{4m}{e \cos \theta} \right)^{\frac{1}{2}} +1 \right]^{m} -1 \right\}.
\end{align*}
This completes the proof of Theorem \ref{th:commutator_esti}. \qed

\appendix
\section{Proof of Lemma \ref{lem:Hermite_monomial}} \label{app:lem_Hermite}

We first prove the following lemma.

\begin{lem} \label{lem:Hermite}
	Let $\omega \in \C$ with $\re \omega \geq 0$, $\alpha = \left( \alpha_{1}, \ldots, \alpha_{n} \right) \in \Z_{\geq 0}^{n}$, and $j \in \left\{ 1, \ldots, n \right\}$.
	Then, the identities
	\begin{align*}
		x_{j} \bm{h}_{\omega, \alpha} \left( x \right) = \begin{cases}
			\omega \bm{h}_{\omega, \alpha +e_{j}} \left( x \right) &\quad \text{if} \quad \alpha_{j} =0, \\
			\omega \bm{h}_{\omega, \alpha +e_{j}} \left( x \right) +2 \alpha_{j} \bm{h}_{\omega, \alpha -e_{j}} \left( x \right) &\quad \text{if} \quad \alpha_{j} \geq 1
		\end{cases}
	\end{align*}
	hold for all $x= \left( x_{1}, \ldots, x_{n} \right) \in \R^{n}$.
\end{lem}

\begin{proof}
	It suffices to show that the identities
	\begin{align*}
		2x_{j} \bm{H}_{\omega, \alpha} \left( x \right) = \begin{cases}
			\omega \bm{H}_{\omega, \alpha +e_{j}} \left( x \right) &\quad \text{if} \quad \alpha_{j} =0, \\
			\omega \bm{H}_{\omega, \alpha +e_{j}} \left( x \right) +2 \alpha_{j} \bm{H}_{\omega, \alpha -e_{j}} \left( x \right) &\quad \text{if} \quad \alpha_{j} \geq 1
		\end{cases}
	\end{align*}
	hold for all $x= \left( x_{1}, \ldots, x_{n} \right) \in \R^{n}$.
	If $\alpha_{j} =0$, then we have
	\begin{align*}
		\omega \bm{H}_{\omega, \alpha +e_{j}} \left( x \right)
		&= \left( -1 \right)^{\abs{\alpha} +1} \omega \exp \left( \frac{\abs{x}^{2}}{\omega} \right) \partial^{\alpha} \partial_{j} \exp \left( - \frac{\abs{x}^{2}}{\omega} \right) \\
		&= \left( -1 \right)^{\abs{\alpha}} \exp \left( \frac{\abs{x}^{2}}{\omega} \right) \partial^{\alpha} \left( 2x_{j} \exp \left( - \frac{\abs{x}^{2}}{\omega} \right) \right) \\
		&=2x_{j} \bm{H}_{\omega, \alpha} \left( x \right).
	\end{align*}
	In the same way, if $\alpha_{j} \geq 1$, then we obtain
	\begin{align*}
		\omega \bm{H}_{\omega, \alpha +e_{j}} \left( x \right)
		&= \left( -1 \right)^{\abs{\alpha}} \exp \left( \frac{\abs{x}^{2}}{\omega} \right) \partial^{\alpha} \left( 2x_{j} \exp \left( - \frac{\abs{x}^{2}}{\omega} \right) \right) \\
		&= \left( -1 \right)^{\abs{\alpha}} \exp \left( \frac{\abs{x}^{2}}{\omega} \right) \left( 2 \alpha_{j} \partial^{\alpha -e_{j}} \exp \left( - \frac{\abs{x}^{2}}{\omega} \right) +2x_{j} \partial^{\alpha} \exp \left( - \frac{\abs{x}^{2}}{\omega} \right) \right) \\
		&=-2 \alpha_{j} \bm{H}_{\omega, \alpha -e_{j}} \left( x \right) +2x_{j} \bm{H}_{\omega, \alpha} \left( x \right).
	\end{align*}
	This completes the proof of Lemma \ref{lem:Hermite}.
\end{proof}

Now, we are ready to show Lemma \ref{lem:Hermite_monomial}.

\begin{proof}[Proof of Lemma \ref{lem:Hermite_monomial}]
	We prove the assertion by induction on $m= \abs{\alpha}$.
	The case where $m=0$, namely, the case where $\alpha =0$ follows from the fact that $\bm{h}_{\omega, 0} \equiv 1$.
	Let $m \in \Z_{\geq 0}$.
	We assume that Lemma \ref{lem:Hermite_monomial} is true for any $\alpha \in \Z_{\geq 0}^{n}$ with $\abs{\alpha} =m$.
	Let $\alpha' \in \Z_{\geq 0}^{n}$ satisfy $\abs{\alpha'} =m+1$ and let $x= \left( x_{1}, \ldots, x_{n} \right) \in \R^{n}$.
	Then, there exist $\alpha = \left( \alpha_{1}, \ldots, \alpha_{n} \right) \in \Z_{\geq 0}^{n}$ with $\abs{\alpha} =m$ and $j \in \left\{ 1, \ldots, n \right\}$ such that $\alpha' = \alpha +e_{j}$.
	If $\alpha_{j} =0$, then by the induction hypothesis and Lemma \ref{lem:Hermite}, we have
	\begin{align*}
		x^{\alpha'} =x_{j} x^{\alpha}
		&= \sum_{2 \beta \leq \alpha} \frac{\alpha !}{\beta ! \left( \alpha -2 \beta \right) !} \omega^{\abs{\alpha - \beta}} x_{j} \bm{h}_{\omega, \alpha -2 \beta} \left( x \right) \\
		&= \sum_{2 \beta \leq \alpha} \frac{\left( \alpha +e_{j} \right) !}{\beta ! \left( \alpha -2 \beta +e_{j} \right) !} \omega^{\abs{\alpha - \beta} +1} \bm{h}_{\omega, \alpha -2 \beta +e_{j}} \left( x \right) \\
		&= \sum_{2 \beta \leq \alpha'} \frac{\alpha' !}{\beta ! \left( \alpha' -2 \beta \right) !} \omega^{\abs{\alpha' - \beta}} \bm{h}_{\omega, \alpha' -2 \beta} \left( x \right).
	\end{align*}
	We next consider the case where $\alpha_{j} \geq 1$.
	If $\alpha_{j}$ is even, then it follows from the induction hypothesis and Lemma \ref{lem:Hermite} that
	\begin{align*}
		x^{\alpha'} =x_{j} x^{\alpha}
		&= \sum_{\substack{2 \beta \leq \alpha \\ 2 \beta_{j} \leq \alpha_{j} -1}} \frac{\alpha !}{\beta ! \left( \alpha -2 \beta \right) !} \omega^{\abs{\alpha - \beta}} x_{j} \bm{h}_{\omega, \alpha -2 \beta} \left( x \right) + \sum_{\substack{2 \beta \leq \alpha \\ 2 \beta_{j} = \alpha_{j}}} \frac{\alpha !}{\beta ! \left( \alpha -2 \beta \right) !} \omega^{\abs{\alpha - \beta}} x_{j} \bm{h}_{\omega, \alpha -2 \beta} \left( x \right) \\
		&= \sum_{\substack{2 \beta \leq \alpha \\ 2 \beta_{j} \leq \alpha_{j} -1}} \frac{\alpha !}{\beta ! \left( \alpha -2 \beta \right) !} \omega^{\abs{\alpha - \beta}} \left( \omega \bm{h}_{\omega, \alpha -2 \beta +e_{j}} \left( x \right) +2 \left( \alpha_{j} -2 \beta_{j} \right) \bm{h}_{\omega, \alpha -2 \beta -e_{j}} \left( x \right) \right) \\
		&\hspace{2cm} + \sum_{\substack{2 \beta \leq \alpha \\ 2 \beta_{j} = \alpha_{j}}} \frac{\alpha !}{\beta ! \left( \alpha -2 \beta \right) !} \omega^{\abs{\alpha - \beta} +1} \bm{h}_{\omega, \alpha -2 \beta +e_{j}} \left( x \right) \\
		&= \sum_{2 \beta \leq \alpha} \frac{\alpha !}{\beta ! \left( \alpha -2 \beta \right) !} \omega^{\abs{\alpha - \beta +e_{j}}} \bm{h}_{\omega, \alpha -2 \beta +e_{j}} \left( x \right) \\
		&\hspace{2cm} +2 \sum_{\substack{2 \beta \leq \alpha +e_{j} \\ \beta_{j} \geq 1}} \frac{\alpha !}{\left( \beta -e_{j} \right) ! \left( \alpha -2 \beta +e_{j} \right) !} \omega^{\abs{\alpha - \beta +e_{j}}} \bm{h}_{\omega, \alpha -2 \beta +e_{j}} \left( x \right) \\
		&= \sum_{\substack{2 \beta \leq \alpha \\ \beta_{j} = 0}} \frac{\alpha !}{\beta ! \left( \alpha -2 \beta \right) !} \omega^{\abs{\alpha' - \beta}} \bm{h}_{\omega, \alpha' -2 \beta} \left( x \right) \\
		&\hspace{2cm} + \sum_{\substack{2 \beta \leq \alpha +e_{j} \\ \beta_{j} \geq 1}} \frac{\alpha !}{\beta ! \left( \alpha -2 \beta +e_{j} \right) !} \left( \left( \alpha_{j} -2 \beta_{j} +1 \right) +2 \beta_{j} \right) \omega^{\abs{\alpha' - \beta}} \bm{h}_{\omega, \alpha' -2 \beta} \left( x \right) \\
		&= \sum_{\substack{2 \beta \leq \alpha +e_{j} \\ \beta_{j} = 0}} \frac{\left( \alpha +e_{j} \right) !}{\beta ! \left( \alpha -2 \beta +e_{j} \right) !} \omega^{\abs{\alpha' - \beta}} \bm{h}_{\omega, \alpha' -2 \beta} \left( x \right) \\
		&\hspace{2cm} + \sum_{\substack{2 \beta \leq \alpha +e_{j} \\ \beta_{j} \geq 1}} \frac{\left( \alpha +e_{j} \right) !}{\beta ! \left( \alpha -2 \beta +e_{j} \right) !} \omega^{\abs{\alpha' - \beta}} \bm{h}_{\omega, \alpha' -2 \beta} \left( x \right) \\
		&= \sum_{2 \beta \leq \alpha'} \frac{\alpha' !}{\beta ! \left( \alpha' -2 \beta \right) !} \omega^{\abs{\alpha' - \beta}} \bm{h}_{\omega, \alpha' -2 \beta} \left( x \right).
	\end{align*}
	Similarly, if $\alpha_{j}$ is odd, then we obtain
	\begin{align*}
		x^{\alpha'} =x_{j} x^{\alpha}
		&= \sum_{2 \beta \leq \alpha -e_{j}} \frac{\alpha !}{\beta ! \left( \alpha -2 \beta \right) !} \omega^{\abs{\alpha - \beta}} x_{j} \bm{h}_{\omega, \alpha -2 \beta} \left( x \right) \\
		&= \sum_{2 \beta \leq \alpha -e_{j}} \frac{\alpha !}{\beta ! \left( \alpha -2 \beta \right) !} \omega^{\abs{\alpha - \beta}} \left( \omega \bm{h}_{\omega, \alpha -2 \beta +e_{j}} \left( x \right) +2 \left( \alpha_{j} -2 \beta_{j} \right) \bm{h}_{\omega, \alpha -2 \beta -e_{j}} \left( x \right) \right) \\
		&= \sum_{2 \beta \leq \alpha -e_{j}} \frac{\alpha !}{\beta ! \left( \alpha -2 \beta \right) !} \omega^{\abs{\alpha - \beta +e_{j}}} \bm{h}_{\omega, \alpha -2 \beta +e_{j}} \left( x \right) \\
		&\hspace{2cm} +2 \sum_{\substack{2 \beta \leq \alpha +e_{j} \\ \beta_{j} \geq 1}} \frac{\alpha !}{\left( \beta -e_{j} \right) ! \left( \alpha -2 \beta +e_{j} \right) !} \omega^{\abs{\alpha - \beta +e_{j}}} \bm{h}_{\omega, \alpha -2 \beta +e_{j}} \left( x \right) \\
		&= \sum_{\substack{2 \beta \leq \alpha -e_{j} \\ \beta_{j} = 0}} \frac{\alpha !}{\beta ! \left( \alpha -2 \beta \right) !} \omega^{\abs{\alpha' - \beta}} \bm{h}_{\omega, \alpha' -2 \beta} \left( x \right) \\
		&\hspace{2cm} + \sum_{\substack{2 \beta \leq \alpha +e_{j} \\ 2 \beta_{j} = \alpha_{j} +1}} \frac{2 \beta_{j} \cdot \alpha !}{\beta! \left( \alpha -2 \beta +e_{j} \right) !} \omega^{\abs{\alpha' - \beta}} \bm{h}_{\omega, \alpha' -2 \beta} \left( x \right) \\
		&\hspace{2cm} + \sum_{\substack{2 \beta \leq \alpha -e_{j} \\ \beta_{j} \geq 1}} \frac{\alpha !}{\beta ! \left( \alpha -2 \beta +e_{j} \right) !} \left( \left( \alpha_{j} -2 \beta_{j} +1 \right) +2 \beta_{j} \right) \omega^{\abs{\alpha' - \beta}} \bm{h}_{\omega, \alpha' -2 \beta} \left( x \right) \\
		&= \sum_{\substack{2 \beta \leq \alpha +e_{j} \\ \beta_{j} = 0}} \frac{\left( \alpha +e_{j} \right) !}{\beta ! \left( \alpha -2 \beta +e_{j} \right) !} \omega^{\abs{\alpha' - \beta}} \bm{h}_{\omega, \alpha' -2 \beta} \left( x \right) \\
		&\hspace{2cm} + \sum_{\substack{2 \beta \leq \alpha +e_{j} \\ \beta_{j} \geq 1}} \frac{\left( \alpha +e_{j} \right) !}{\beta ! \left( \alpha -2 \beta +e_{j} \right) !} \omega^{\abs{\alpha' - \beta}} \bm{h}_{\omega, \alpha' -2 \beta} \left( x \right) \\
		&= \sum_{2 \beta \leq \alpha'} \frac{\alpha' !}{\beta ! \left( \alpha' -2 \beta \right) !} \omega^{\abs{\alpha' - \beta}} \bm{h}_{\omega, \alpha' -2 \beta} \left( x \right).
	\end{align*}
	This completes the induction argument.
\end{proof}

\section{Application to the complex Ginzburg--Landau equation} \label{app:CGL_weight}

In this section, as an application of Theorems \ref{th:commutator} and \ref{th:commutator_esti}, we show the following proposition on a weighted estimate of global solutions to \eqref{CGL} with $\mu =0$ in the super Fujita-critical case: $p>1+2/n$.

\begin{prop} \label{pro:CGL_weight}
	Let $\nu \in \C$ with $\re \nu >0$, $p>1+2/n$, and $u_{0} \in \left( L^{1} \cap L^{\infty} \right) \left( \R^{n} \right)$.
	Let $q \in \left[ 1, + \infty \right]$ and let $m \in \Z_{>0}$.
	Assume that $u \in C \left( \left( 0, + \infty \right); \left( L^{1} \cap L^{\infty} \right) \left( \R^{n} \right) \right)$ is a global solution to \eqref{CGL} with $\mu =0$ satisfying
	\begin{align}
		\label{CGL_decay}
		\sup_{r \in \left[ 1, + \infty \right]} \sup_{t>0} \left( 1+t \right)^{\frac{n}{2} \left( 1- \frac{1}{r} \right)} \norm{u \left( t \right)}_{r} <+ \infty.
	\end{align}
	If $u_{0} \in L_{m}^{q} \left( \R^{n} \right)$, then $u \in C \left( \left( 0, + \infty \right); L_{m}^{q} \left( \R^{n} \right) \right)$ with the estimate
	\begin{align*}
		\sum_{\abs{\alpha} =m} \norm{x^{\alpha} u \left( t \right)}_{q} \leq C \left( 1+t^{\frac{m}{2}} \right)
	\end{align*}
	for all $t>0$.
\end{prop}

\begin{rem}
	If $p>1+2/n$ and $\norm{u_{0}}_{1} + \norm{u_{0}}_{\infty}$ is sufficiently small, then a global solution to \eqref{CGL} with $\mu =0$ satisfying \eqref{CGL_decay} can be constructed by the standard contraction argument to the following integral equation associated with \eqref{CGL}:
	\begin{align}
		\label{CGL_integral}
		u \left( t \right) =e^{t \nu \Delta} u_{0} + \int_{0}^{t} e^{\left( t-s \right) \nu \Delta} f \left( u \left( s \right) \right) ds,
	\end{align}
	where $f \left( \xi \right) \coloneqq \lambda \abs{\xi}^{p-1} \xi$.
	We remark that the smallness of $\norm{\abs{x}^{m} u_{0}}_{q}$ is not assumed in Proposition \ref{pro:CGL_weight}, which is one of the advantages to employ the commutation relations.
\end{rem}

In order to prove Proposition \ref{pro:CGL_weight}, we need to prepare three lemmas.

\begin{lem}[{cf. \cite[Lemma 2.1]{Kusaba-Ozawa}}] \label{lem:LpLq}
	Let $\omega \in \C$ satisfy $\re \omega >0$.
	Let $1 \leq q \leq p \leq + \infty$ and let $\alpha \in \Z_{\geq 0}^{n}$.
	Then, the estimate
	\begin{align*}
		\norm{\partial^{\alpha} e^{\omega \Delta} \varphi}_{p} \leq \abs{\omega}^{- \frac{n}{2} \left( \frac{1}{q} - \frac{1}{p} \right) - \frac{\abs{\alpha}}{2}} \norm{\partial^{\alpha} G_{e^{i \theta}}}_{r} \norm{\varphi}_{q}
	\end{align*}
	holds for any $\varphi \in L^{q} \left( \R^{n} \right)$, where $\theta \in \left( - \pi /2, \pi /2 \right)$ with $e^{i \theta} = \omega / \abs{\omega}$ and $r \in \left[ 1, + \infty \right]$ with $1/p+1=1/r+1/q$.
\end{lem}

\begin{lem} \label{lem:application}
	Let $\omega \in \C$ satisfy $\re \omega >0$.
	Let $1 \leq q \leq p \leq + \infty$, $m \in \Z_{>0}$, $j \in \left\{ 1, \ldots, n \right\}$, and $\alpha \in \Z_{\geq 0}^{n}$ with $\abs{\alpha} =m$.
	Then, for any $\varphi \in L_{m}^{q} \left( \R^{n} \right)$, the identity
	\begin{align}
		\label{eq:B1}
		x_{j} R_{\alpha} \left( \omega \right) \varphi =R_{\alpha +e_{j}} \left( \omega \right) \varphi -R_{e_{j}} \left( \omega \right) \left( x^{\alpha} \varphi \right)
	\end{align}
	holds in $L^{q} \left( \R^{n} \right)$.
	Moreover, the estimate
	\begin{align}
		\label{eq:B2}
		&\sum_{j=1}^{n} \sum_{\abs{\alpha} =m} \norm{x_{j} R_{\alpha} \left( \omega \right) \varphi}_{p} \nonumber \\
		&\hspace{1cm} \leq \left( nA_{m+1, r} \left( \theta \right) + \frac{\left( n+m-1 \right) !}{\left( n-1 \right) ! m!} A_{1, r} \left( \theta \right) \right) \abs{\omega}^{- \frac{n}{2} \left( \frac{1}{q} - \frac{1}{p} \right)} \left( \abs{\omega}^{\frac{1}{2}} \norm{\abs{x}^{m} \varphi}_{q} + \abs{\omega}^{\frac{m+1}{2}} \norm{\varphi}_{q} \right)
	\end{align}
	holds for all $\varphi \in L_{m}^{q} \left( \R^{n} \right)$, where $\theta \in \left( - \pi /2, \pi /2 \right)$ with $e^{i \theta} = \omega / \abs{\omega}$ and $r \in \left[ 1, + \infty \right]$ with $1/p+1=1/r+1/q$.
\end{lem}

\begin{rem}
	If $\varphi \in L^{q}_{m+1} \left( \R^{n} \right)$, then it follows from Theorem \ref{th:commutator} that
	\begin{align*}
		x_{j} R_{\alpha} \left( \omega \right) \varphi &=x_{j} \left[ x^{\alpha}, e^{\omega \Delta} \right] \varphi \\
		&= \left[ x_{j} x^{\alpha}, e^{\omega \Delta} \right] \varphi - \left[ x_{j}, e^{\omega \Delta} \right] \left( x^{\alpha} \varphi \right) \\
		&=R_{\alpha +e_{j}} \left( \omega \right) \varphi -R_{e_{j}} \left( \omega \right) \left( x^{\alpha} \varphi \right)
	\end{align*}
	for any $j \in \left\{ 1, \ldots, n \right\}$ and $\alpha \in \Z_{\geq 0}^{n}$ with $\abs{\alpha} =m$.
	However, if $\varphi \in L^{q}_{m} \left( \R^{n} \right) \setminus L^{q}_{m+1} \left( \R^{n} \right)$, then the terms $\left[ x_{j} x^{\alpha}, e^{\omega \Delta} \right] \varphi$ and $\left[ x_{j}, e^{\omega \Delta} \right] \left( x^{\alpha} \varphi \right)$ in the second identity are not well-defined in $L^{q} \left( \R^{n} \right)$.
	For this reason, to derive \eqref{eq:B1} for all $\varphi \in L^{q}_{m} \left( \R^{n} \right)$, we need to compute $x_{j} R_{\alpha} \left( \omega \right) \varphi$ directly by using the definition of $R_{\alpha} \left( \omega \right) \varphi$.
\end{rem}

\begin{proof}[Proof of Lemma \ref{lem:application}]
	Let $\varphi \in L_{m}^{q} \left( \R^{n} \right)$.
	By \eqref{eq:R_conv}, we have
	\begin{align*}
		&x_{j} R_{\alpha} \left( \omega \right) \varphi +R_{e_{j}} \left( \omega \right) \left( x^{\alpha} \varphi \right) \\
		&\hspace{1cm} = \sum_{\substack{\beta + \gamma = \alpha \\ \beta \neq 0}} \frac{\alpha !}{\beta ! \gamma !} {\,} \left( ( x_{j} x^{\beta} G_{\omega} ) \ast \left( x^{\gamma} \varphi \right) + ( x^{\beta} G_{\omega} ) \ast \left( x_{j} x^{\gamma} \varphi \right) \right) + \left( x_{j} G_{\omega} \right) \ast \left( x^{\alpha} \varphi \right) \\
		&\hspace{1cm} = \sum_{\substack{\beta' + \gamma' = \alpha +e_{j} \\ \gamma' \leq \alpha}} \frac{\alpha !}{\left( \beta' -e_{j} \right) ! \gamma' !} {\,} ( x^{\beta'} G_{\omega} ) \ast ( x^{\gamma'} \varphi ) + \sum_{\substack{\beta' + \gamma' = \alpha +e_{j} \\ 0 \neq \beta' \leq \alpha}} \frac{\alpha !}{\beta' ! \left( \gamma' -e_{j} \right) !} {\,} ( x^{\beta'} G_{\omega} ) \ast ( x^{\gamma'} \varphi ) \\
		&\hspace{1cm} = \sum_{\substack{\beta' + \gamma' = \alpha +e_{j} \\ \beta' \neq 0}} \frac{\left( \alpha +e_{j} \right) !}{\beta' ! \gamma' !} {\,} ( x^{\beta'} G_{\omega} ) \ast ( x^{\gamma'} \varphi ) \\
		&\hspace{1cm} =R_{\alpha + e_{j}} \left( \omega \right) \varphi.
	\end{align*}
	The estimate \eqref{eq:B2} can be shown in the same way as in the proof of Theorem \ref{th:commutator_esti}.
\end{proof}

\begin{lem} \label{lem:commutator_gene}
	Let $\omega \in \C$ satisfy $\re \omega >0$.
	Let $\eta \in W^{1, \infty} \left( \R^{n} \right)$ and let $1 \leq q \leq p \leq + \infty$.
	Then, the estimate
	\begin{align*}
		\norm{\left[ \eta, e^{\omega \Delta} \right] \varphi}_{p} \leq A_{1, r} \left( \theta \right) \abs{\omega}^{- \frac{n}{2} \left( \frac{1}{q} - \frac{1}{p} \right) + \frac{1}{2}} \norm{\nabla \eta}_{\infty} \norm{\varphi}_{q}
	\end{align*}
	holds for any $\varphi \in L^{q} \left( \R^{n} \right)$, where $\left[ \eta, e^{\omega \Delta} \right] \varphi \coloneqq \eta e^{\omega \Delta} \varphi -e^{\omega \Delta} \left( \eta \varphi \right)$, $\theta \in \left( - \pi /2, \pi /2 \right)$ with $e^{i \theta} = \omega / \abs{\omega}$, and $r \in \left[ 1, + \infty \right]$ with $1/p+1=1/r+1/q$.
\end{lem}

\begin{proof}
	For any $x \in \R^{n}$, we have
	\begin{align*}
		\left( \left[ \eta, e^{\omega \Delta} \right] \varphi \right) \left( x \right)
		&= \int_{\R^{n}} \left( \eta \left( x \right) - \eta \left( y \right) \right) G_{\omega} \left( x-y \right) \varphi \left( y \right) dy \\
		&= \int_{\R^{n}} \left( \int_{0}^{1} \frac{d}{d \theta} \eta \left( y+ \theta \left( x-y \right) \right) d \theta \right) G_{\omega} \left( x-y \right) \varphi \left( y \right) dy \\
		&= \int_{\R^{n}} \int_{0}^{1} \left( \nabla \eta \right) \left( y+ \theta \left( x-y \right) \right) \cdot \left( x-y \right) G_{\omega} \left( x-y \right) \varphi \left( y \right) d \theta dy.
	\end{align*}
	By Young's inequality, we obtain
	\begin{align*}
		\norm{\left[ \eta, e^{\omega \Delta} \right] \varphi}_{p}
		&\leq \norm{\nabla \eta}_{\infty} \norm{\abs{\abs{x} G_{\omega}} \ast \abs{\varphi}}_{p} \\
		&\leq \norm{\nabla \eta}_{\infty} \norm{\abs{x} G_{\omega}}_{r} \norm{\varphi}_{q} \\
		&= \abs{\omega}^{- \frac{n}{2} \left( \frac{1}{q} - \frac{1}{p} \right) + \frac{1}{2}} \norm{\nabla \eta}_{\infty} \norm{\abs{x} G_{e^{i \theta}}}_{r} \norm{\varphi}_{q} \\
		&\leq A_{1, r} \left( \theta \right) \abs{\omega}^{- \frac{n}{2} \left( \frac{1}{q} - \frac{1}{p} \right) + \frac{1}{2}} \norm{\nabla \eta}_{\infty} \norm{\varphi}_{q}.
	\end{align*}
	Here, we have used the fact that $\norm{\abs{x} G_{e^{i \theta}}}_{r} \leq A_{1, r} \left( \theta \right)$.
\end{proof}

We are ready to show Proposition \ref{pro:CGL_weight}.

\begin{proof}[Proof of Proposition \ref{pro:CGL_weight}]
	We prove Proposition \ref{pro:CGL_weight} by induction on $m \in \Z_{>0}$.
	For $j \in \left\{ 1, \ldots, n \right\}$ and $\varepsilon >0$, we define a function $\eta_{j, \varepsilon} \colon \R^{n} \to \R$ by
	\begin{align*}
		\eta_{j, \varepsilon} \left( x \right) \coloneqq x_{j} e^{- \varepsilon \abs{x}^{2}}, \qquad x= \left( x_{1}, \ldots, x_{n} \right) \in \R^{n}.
	\end{align*}
	Then, we see that $\eta_{j, \varepsilon} \in W^{1, \infty} \left( \R^{n} \right)$ and
	\begin{align}
		\label{eq:weight_app}
		\norm{\nabla \eta_{j, \varepsilon}}_{\infty} \leq 2 \sup_{\rho \geq 0} \rho e^{- \rho} +1 \leq 2.
	\end{align}

	We assume that Proposition \ref{pro:CGL_weight} holds for some $m \in \Z_{>0}$ and that $u_{0} \in L_{m+1}^{q} \left( \R^{n} \right)$.
	Let $t>0$ and let $\alpha' \in \Z_{\geq 0}^{n}$ satisfy $\abs{\alpha'} =m+1$.
	Then, there exist $\alpha \in \Z_{\geq 0}^{n}$ with $\abs{\alpha} =m$ and $j \in \left\{ 1, \ldots, n \right\}$ such that $\alpha' = \alpha +e_{j}$.
	Multiplying \eqref{CGL_integral} by $\eta_{j, \varepsilon} x^{\alpha}$ and applying Theorem \ref{th:commutator}, we obtain
	\begin{align*}
		\eta_{j, \varepsilon} x^{\alpha} u \left( t \right)
		&= \eta_{j, \varepsilon} x^{\alpha} e^{t \nu \Delta} u_{0} + \int_{0}^{t} \eta_{j, \varepsilon} x^{\alpha} e^{\left( t-s \right) \nu \Delta} f \left( u \left( s \right) \right) ds \\
		&=e^{t \nu \Delta} \left( \eta_{j, \varepsilon} x^{\alpha} u_{0} \right) + \left[ \eta_{j, \varepsilon}, e^{t \nu \Delta} \right] \left( x^{\alpha} u_{0} \right) + \eta_{j, \varepsilon} R_{\alpha} \left( t \nu \right) u_{0} \\
		&\hspace{1cm} + \int_{0}^{t} e^{\left( t-s \right) \nu \Delta} \left( \eta_{j, \varepsilon} x^{\alpha} f \left( u \left( s \right) \right) \right) ds \\
		&\hspace{1cm} + \int_{0}^{t} \left[ \eta_{j, \varepsilon}, e^{\left( t-s \right) \nu \Delta} \right] \left( x ^{\alpha} f \left( u \left( s \right) \right) \right) ds \\
		&\hspace{1cm} + \int_{0}^{t} \eta_{j, \varepsilon} R_{\alpha} \left( \left( t-s \right) \nu \right) f \left( u \left( s \right) \right) ds,
	\end{align*}
	where $f \left( \xi \right) \coloneqq \lambda \abs{\xi}^{p-1} \xi$.
	By \eqref{CGL_decay}, \eqref{eq:weight_app}, Lemmas \ref{lem:LpLq}, \ref{lem:application} and \ref{lem:commutator_gene}, and the induction hypothesis, we have
	\begin{align*}
		\norm{\eta_{j, \varepsilon} x^{\alpha} u \left( t \right)}_{q}
		&\leq \norm{e^{t \nu \Delta} \left( \eta_{j, \varepsilon} x^{\alpha} u_{0} \right)}_{q} + \norm{\left[ \eta_{j, \varepsilon}, e^{t \nu \Delta} \right] \left( x^{\alpha} u_{0} \right)}_{q} + \norm{x_{j} R_{\alpha} \left( t \nu \right) u_{0}}_{q} \\
		&\hspace{1cm} + \int_{0}^{t} \norm{e^{\left( t-s \right) \nu \Delta} \left( \eta_{j, \varepsilon} x^{\alpha} f \left( u \left( s \right) \right) \right)}_{q} ds \\
		&\hspace{1cm} + \int_{0}^{t} \norm{\left[ \eta_{j, \varepsilon}, e^{\left( t-s \right) \nu \Delta} \right] \left( x ^{\alpha} f \left( u \left( s \right) \right) \right)}_{q} ds \\
		&\hspace{1cm} + \int_{0}^{t} \norm{x_{j} R_{\alpha} \left( \left( t-s \right) \nu \right) f \left( u \left( s \right) \right)}_{q} ds \\
		&\leq C \norm{\eta_{j, \varepsilon} x^{\alpha} u_{0}}_{q} +Ct^{\frac{1}{2}} \norm{\nabla \eta_{j, \varepsilon}}_{\infty} \norm{x^{\alpha} u_{0}}_{q} +C \left( t^{\frac{1}{2}} \norm{\abs{x}^{m} u_{0}}_{q} +t^{\frac{m+1}{2}} \norm{u_{0}}_{q} \right) \\
		&\hspace{1cm} +C \int_{0}^{t} \norm{u \left( s \right)}_{\infty}^{p-1} \norm{\eta_{j, \varepsilon} x^{\alpha} u \left( s \right)}_{q} ds \\
		&\hspace{1cm} +C \int_{0}^{t} \left( t-s \right)^{\frac{1}{2}} \norm{\nabla \eta_{j, \varepsilon}}_{\infty} \norm{u \left( s \right)}_{\infty}^{p-1} \norm{x ^{\alpha} u \left( s \right)}_{q} ds \\
		&\hspace{1cm} +C \int_{0}^{t} \left( \left( t-s \right)^{\frac{1}{2}} \norm{\abs{x}^{m} u \left( s \right)}_{q} + \left( t-s \right)^{\frac{m+1}{2}} \norm{u \left( s \right)}_{q} \right) \norm{u \left( s \right)}_{\infty}^{p-1} ds \\
		&\leq C \norm{x_{j} x^{\alpha} u_{0}}_{q} +Ct^{\frac{1}{2}} \norm{x^{\alpha} u_{0}}_{q} +C \left( t^{\frac{1}{2}} \norm{\abs{x}^{m} u_{0}}_{q} +t^{\frac{m+1}{2}} \norm{u_{0}}_{q} \right) \\
		&\hspace{1cm} +C \int_{0}^{t} \left( 1+s \right)^{- \frac{n}{2} \left( p-1 \right)} \norm{\eta_{j, \varepsilon} x^{\alpha} u \left( s \right)}_{q} ds \\
		&\hspace{1cm} +Ct^{\frac{1}{2}} \int_{0}^{t} \left( 1+s \right)^{- \frac{n}{2} \left( p-1 \right)} \left( 1+s^{\frac{m}{2}} \right) ds \\
		&\hspace{1cm} +C \int_{0}^{t} \left( t^{\frac{1}{2}} \left( 1+s^{\frac{m}{2}} \right) +t^{\frac{m+1}{2}} \left( 1+s \right)^{- \frac{n}{2} \left( 1- \frac{1}{q} \right)} \right) \left( 1+s \right)^{- \frac{n}{2} \left( p-1 \right)} ds \\
		&\leq C \left( 1+t^{\frac{m+1}{2}} \right) +C \int_{0}^{t} \left( 1+s \right)^{- \frac{n}{2} \left( p-1 \right)} \norm{\eta_{j, \varepsilon} x^{\alpha} u \left( s \right)}_{q} ds.
	\end{align*}
	Therefore, it follows from the Gr\"{o}nwall lemma that
	\begin{align*}
		\norm{\eta_{j, \varepsilon} x^{\alpha} u \left( t \right)}_{q}
		&\leq C \left( 1+t^{\frac{m+1}{2}} \right) \exp \left( C \int_{0}^{t} \left( 1+s \right)^{- \frac{n}{2} \left( p-1 \right)} ds \right) \\
		&\leq C \left( 1+t^{\frac{m+1}{2}} \right).
	\end{align*}
	By taking $\varepsilon \searrow 0$ and applying Fatou's lemma, we deduce that $x^{\alpha'} u \left( t \right) =x_{j} x^{\alpha} u \left( t \right) \in L^{q} \left( \R^{n} \right)$ and
	\begin{align*}
		\lVert x^{\alpha'} u \left( t \right) \rVert_{q} &\leq C \left( 1+t^{\frac{m+1}{2}} \right).
	\end{align*}
	This in turn implies $x^{\alpha'} f \left( u \left( t \right) \right) \in L^{q} \left( \R^{n} \right)$.
	Moreover, by \eqref{CGL_integral} and Theorem \ref{th:commutator}, we get
	\begin{align*}
		x^{\alpha'} u \left( t \right)
		&=x^{\alpha'} e^{t \nu \Delta} u_{0} + \int_{0}^{t} x^{\alpha'} e^{\left( t-s \right) \nu \Delta} f \left( u \left( s \right) \right) ds \\
		&=e^{t \nu \Delta} {\,} ( x^{\alpha'} u_{0} ) +R_{\alpha'} \left( t \nu \right) u_{0} + \int_{0}^{t} e^{\left( t-s \right) \nu \Delta} {\,} ( x^{\alpha'} f \left( u \left( s \right) \right) ) {\,} ds+ \int_{0}^{t} R_{\alpha'} \left( \left( t-s \right) \nu \right) f \left( u \left( s \right) \right) ds,
	\end{align*}
	whence follows $x^{\alpha'} u \in C \left( \left( 0, + \infty \right); L^{q} \left( \R^{n} \right) \right)$.
	In the same way, we can prove Proposition \ref{pro:CGL_weight} with $m=1$.
\end{proof}

\section*{Acknowledgments}

The first author was supported by JSPS Invitational Fellowships for Research in Japan \# S24040 and he would like to thank Hatem Zaag for continuous encouragement over the years.
The second author was supported by JSPS Grant-in-Aid for JSPS Fellows \# JP24KJ2084.
The third author was supported by JSPS Grant-in-Aid for Scientific Research (S) \# JP24H00024.


%


\end{document}